\documentclass[12pt,reqno]{amsart}
\usepackage[utf8]{inputenc}
\usepackage[T1]{fontenc}
\usepackage[usenames, dvipsnames]{color}
\usepackage{ulem}

\usepackage{dsfont, amsfonts, amsmath, amssymb,amscd, stmaryrd, latexsym, amsthm, dsfont}
\usepackage[frenchb,english]{babel}
\usepackage{enumerate}
\usepackage{longtable}
\usepackage{geometry}
\usepackage{float}
\usepackage{tikz}
\usetikzlibrary{shapes,arrows}
\usepackage{float}
\usepackage{tabularx}
\usepackage{tikz}
\usepackage{xypic}
\usetikzlibrary{shapes,arrows}

\newtheorem{theorem}{Theorem}[section]
\newtheorem{lemma}[theorem]{Lemma}
\newtheorem{proposition}[theorem]{Proposition}
\newtheorem{corollary}[theorem]{Corollary}

\theoremstyle{remark}
\newtheorem{remark}[theorem]{\bf Remark}

\usepackage{hyperref}
\hypersetup{
	colorlinks=true,
	urlcolor=blue,
	citecolor=blue}
\def\NN{\mathds{N}}
\def\RR{\mathbb{R}}
\def\QQ{\mathbb{Q}}
\def\CC{\mathbb{C}}
\def\ZZ{\mathbb{Z}}
\def\kk{\mathds{k}}
\begin{document}
	
\def\NN{\mathbb{N}}
\def\RR{\mathds{R}}
\def\HH{I\!\! H}
\def\QQ{\mathbb{Q}}
\def\CC{\mathds{C}}
\def\ZZ{\mathbb{Z}}
\def\DD{\mathds{D}}
\def\OO{\mathcal{O}}
\def\kk{\mathds{k}}
\def\KK{\mathbb{K}}
\def\ho{\mathcal{H}_0^{\frac{h(d)}{2}}}
\def\LL{\mathbb{L}}
\def\L{\mathds{k}_2^{(2)}}
\def\M{\mathds{k}_2^{(1)}}
\def\k{\mathds{k}^{(*)}}
\def\l{\mathds{L}}

\selectlanguage{english}
 \title[On the $2$-class group...]{On the $2$-class group of some  number fields with
 	large degree}
 \author[M. M. Chems-Eddin]{Mohamed Mahmoud CHEMS-EDDIN}
 \address{Mohamed Mahmoud CHEMS-EDDIN: Mohammed First University, Mathematics Department, Sciences Faculty, Oujda, Morocco }
 \email{2m.chemseddin@gmail.com}

 \author[A. Azizi]{Abdelmalek Azizi}
 \address{Abdelmalek Azizi: Mohammed First University, Mathematics Department, Sciences Faculty, Oujda, Morocco }
 \email{abdelmalekazizi@yahoo.fr}

 \author[A. Zekhnini]{Abdelkader Zekhnini}
 \address{Abdelkader Zekhnini: Mohammed First University, Mathematics Department, Pluridisciplinary faculty, Nador, Morocco}
 \email{zekha1@yahoo.fr}


\begin{abstract}
	Let $d$ be an odd   square-free integer, $m\geq 3$  any integer  and $L_{m, d}:=\mathbb{Q}(\zeta_{2^m},\sqrt{d})$.
In this paper, we shall determine all the fields $L_{m, d}$ having an odd class number. Furthermore, using the
cyclotomic $\mathbb{Z}_2$-extensions of some number fields, we compute the rank of the $2$-class group of $L_{m, d}$   whenever the prime divisors of $d$ are congruent to $3$ or  $5\pmod 8$.
\end{abstract}

\selectlanguage{english}

\maketitle
\section{\textbf{Introduction}}
Let $K$ be an algebraic  number field. For a prime integer $p$, let  $\mathrm{Cl}_p(K)$ denote the  $p$-class group of $K$, that is the $p$-Sylow subgroup  of its ideal class group $\mathrm{Cl}(K)$ in the wide sense. The class group $\mathrm{Cl}(K)$, its subgroup $\mathrm{Cl}_p(K)$ and  their orders and structures  have been investigated  and studied in many papers for a long time, and there are   many interesting open problems related to these topics  which are the object of
intense studies.

One  classical and difficult problem in algebraic number theory is the determination of the rank of the $p$-class group of a number field $K$. When $p=2$ and $K$ is a quadratic extension of a number field $k$ having an odd class number, the ambiguous class number formula can be used    to determine this rank, involving   units of $k$ which are norms in $K/k$ and ramified primes in $K/k$ (cf. \cite{Gr}). This fact is practically one of the most important means for   structuring the $2$-class group of a given number field of small degree (cf. \cite{chemsZkhnin1,mccall19972}). Our contribution in this article is
to study the $2$-rank  of an infinite  family of  number fields, with large degree over $\mathbb{Q}$. Comparing with other   papers tackling this problem, the main novelty of this article is the combination  of ramification theory, ambiguous class number formula and the theory of
cyclotomic $\mathbb{Z}_2$-extensions of some number fields.

Let $d$   be an odd     square-free integer,  $m\geq 3$  any integer and $L_{m,d}:=\mathbb{Q}(\zeta_{2^m},\sqrt{d})$.
In the present paper, we are interested in studying  the parity of the class number of all  the fields  $L_{m,d}$. Furthermore, we compute the rank of the $2$-class group of $L_{m,d}$   assuming  the prime divisors of $d$ are congruent to $3$ or  $5\pmod 8$. Since the unit group of  $\mathbb{Q}(\zeta_{ 2^{m}})$, with $m\geq7$, is not described until today, the  methods using  unit  groups for computing the rank of the $2$-class group  of a given number field  are not valid for treating such problem in our case when $m\geq 7$. For this, we will call some results from Iwasawa theory to overcome    the problem. In the appendix, we compute the rank of the $2$-class group of $L_{m,d}^+$, the maximal real subfield of $L_{m,d}$, in terms of the number of prime divisors of $d$.

Finally, to sum up, let us highlight the importance of some parts of the present work. Note that  the layers of the  $\ZZ_2$-extensions of  $\kk=\mathbb{Q}(\sqrt{d}, \sqrt{-1})$ were     subject of some recent studies (e.g. \cite{Hubbard,LiouYangXuZhang}); and in this paper, we give   more arithmetical properties of  $L_{m,d}$ (resp. $L_{m,d}^+$), the layers of the cyclotomic $\ZZ_2$-extension  of $\kk$  (resp. $ \mathbb{Q}(\sqrt{d} )$).
Furthermore, we discuss the interesting question of the parity  of the class number  of $L_{m,d}$, and we explicitly give the rank of its $2$-class group which is strongly related to   the interesting problem of the structure of the Iwasawa module (see for example Corollary \ref{cor 4.5} or \cite{chemskatharina}).  The authors of \cite{chemskatharina} used this paper with some other techniques    of Iwasawa theory to determine the structure of the $2$-class group of some fields $L_{m,d}$.

 Before  quoting some preliminary results, let us fix the following notations which will be used throughout this paper.

\section*{Notations}
{\begin{enumerate}[\rm$\star$]
	
		\item $d$:   An odd square-free integer,
		\item  $m$:   A positive integer $\geq 3$,
		\item  $\zeta_{m}$:   An nth primitive root of unity,
		\item  $K_m=$    $\mathbb{Q}(\zeta_{2^m})$,
		\item  $L_{m,d}=$    $K_m(\sqrt{d})$,
				\item  $k^+$:  The maximal real subfield of a number field $k$,
		\item  $\mathrm{Cl}_2(k)$:  The $2$-class group of a number field  $k$,
		 \item  $k_\infty$:  The $\mathbb{Z}_2$-extension of a number field $k$,
		 \item  $k_n$:  The nth layer of  $k_\infty/k$,
		 \item  $X_\infty$: $\varprojlim (\mathrm{Cl}_2(k_n))$,
		 \item $\mathcal O_{k}$:  The ring of integers of  $k$,
		 \item  $h(k)$:  The class number of  $k$,
		 \item  $h_2(k)$:  The $2$-class number of  $k$,
		 \item  $N$:   The   norm map of the extension $L_{m, d}/K_m$,
		 \item  $E_{k}$:   The unit group of $k$,
		 \item  $e_{m,d}$:  Defined by $(E_{K_m}:E_{K_m}\cap N(L_{m,d}))=2^{ e_{m,d}}$,
		 \item  $\left( \frac{\alpha,d}{\mathfrak p}\right)$:  The quadratic  norm residue symbol 	for $L_{m,d}/K_m$,
		 \item  $\varepsilon_l$:  The fundamental unit of  the quadratic   field $\mathbb{Q}(\sqrt l)$,
		 \item  $h_2(d)$:  The $2$-class number of the quadratic   field $\mathbb{Q}(\sqrt{d})$,
		 \item $rank_2(Cl(L_{m,d}))$:  The rank of the $2$-class group of $L_{m,d}$.	
\end{enumerate}}

 \section{\textbf{Preliminary results}}
Let us collect some results that will be used in the sequel.  Let $k$ be an  algebraic  number field and $k_\infty$  a $\mathbb{Z}_2$-extension of $k$, that is a Galois extension of $k$ whose Galois group is topologically isomorphic to the $2$-adic ring $\mathbb{Z}_2$. For a  non-negative integer $n$, denote by $k_n$ the  intermediate field of $k_\infty/k$ with degree $2^n$ over $k$. Begin by the following theorem which deals with ranks and class numbers of the intermediate subextensions of $k_\infty/k$.

  	\begin{theorem}[\cite{fukuda}]\label{lm fukuda}
  	Let $k_\infty/k$ be a $\mathbb{Z}_2$-extension and $n_0$  an integer such that any prime of $k_\infty$ which is ramified in $k_\infty/k$ is totally ramified in $k_\infty/k_{n_0}$.
  	\begin{enumerate}[\rm 1.]
  		\item If there exists an integer $n\geq n_0$ such that   $h_2(k_n)=h_2(k_{n+1})$, then $h_2(k_n)=h_2(k_{m})$ for all $m\geq n$.
  		\item If there exists an integer $n\geq n_0$ such that $rank_2( \mathrm{Cl}(k_n))= rank_2( \mathrm{Cl}(k_{n+1}))$, then
  		$rank_2(\mathrm{Cl}(k_{m}))= rank_2(\mathrm{Cl}(k_{n}))$ for all $m\geq n$.
  	\end{enumerate}
  \end{theorem}

 	\begin{theorem}[\text{\cite[Theorem 10.1]{washington1997introduction}}]\label{thm h(k) divise h(L)}
 	If an extension of number fields $L/K$ contains no unramified abelian   subextensions   $F/K$, with $F\not=K$, then $h(K)$ divides $h(L)$.
 \end{theorem}

  \begin{lemma}[\text{\cite[Lamma 8.1]{washington1997introduction}}]\label{lm cyclo units}
  	The cyclotomic units of $K_m$$($resp.  $K_m^+$$)$ are generated by $\zeta_{ 2^{m}}$$($resp.  $-1$$)$ and $\xi_{k,m}=\zeta_{2^m}^{(1-k)/2}\frac{1-\zeta_{2^m}^k}{ 1-\zeta_{2^m}},$ where $k$ is an odd integer such that  $1< k< 2^{m-1}$.
  \end{lemma}
The following result is a consequence of ramification theory in a Kummer extension.
  \begin{theorem}[\text{\cite{d.hilbet}}]
  	Let $K/k$ be a quadratic extension and $\mu\in k$ prime to $2$ such that $K=k(\sqrt{\mu})$. The extension $K/k$ is unramified at  finite primes if and only if $\mu$ verifies the following properties:
  	\begin{enumerate}[\rm 1.]
  		\item The principal ideal generated by $\mu$ is a square of a fractional ideal of $k$.
  		\item There exists $ \xi\in k$ such that $\mu\equiv  \xi^2\pmod 4$.
  	\end{enumerate}
  \end{theorem}

  	\begin{lemma}[\cite{chemsZkhnin1}]\label{lemma 1 values of symbls p=3 mod 8}
  Let $p$ be a   prime integer and  $\mathfrak p_{K_3}$ a prime ideal of $K_3 $ lying over $p$.
  	\begin{enumerate}[\rm 1.]
  		\item If  $p\equiv 3\mod 8$, then
  		$\left( \frac{\zeta_{8}, p}{ \mathfrak p_{K_3}}\right)=-1 \text{ and } \left( \frac{ \varepsilon_2, p}{\mathfrak p_{ K_3}}\right)=-1
  		.$
  		\item If   $p\equiv 5\mod 8$,  then
  		$\left( \frac{\zeta_{8}, p}{  \mathfrak p_{K_3}}\right)=-1\text{  and }\left( \frac{ \varepsilon_2, p}{\mathfrak p_{K_3}}\right)=1.
  		$	
  	\end{enumerate}
  \end{lemma}



  \begin{proposition}\label{prop decoposition}
  	Let $m \geq 3$ be an integer and $d$  an odd square-free integer. The ring of integers of $L_{m, d}$ is given by   $$	\mathcal{O}_{L_{m, d}}=\left\{  \begin{array}{ccc}
  	\mathbb{Z}[\zeta_{2^m}, \frac{1+\sqrt{d}}{2}]& \text{ if}& d\equiv 1\mod 4,\\
  	\mathbb{Z}[\zeta_{2^m}, \frac{1+\sqrt{-d}}{2}]& \text{ if}& d\equiv 3\mod 4.
  	\end{array}
  	\right.$$
  	
  	Furthermore,  the relative  discriminant of $L_{m, d}/K_m$ is $\delta_{L_{m, d}/K_m}= d\mathcal{O}_{L_{m, d}}$.
  \end{proposition}
  \begin{proof}
  	Assume that  $d\equiv 1\pmod 4$,  then  $\delta_{K_m}\wedge \delta_{\mathbb{Q}(\sqrt{d})}=1$,  so $\mathcal{O}_{L_{m, d}}=\mathcal{O}_{K_{m}}\mathcal{O}_{\mathbb Q(\sqrt{d})}=\mathbb{Z}[\zeta_{2^m}, \frac{1+\sqrt{d}}{2}]=\mathcal{O}_{K_{m}}[\frac{1+\sqrt{d}}{2}].$ So the relative discriminant of  $L_{m, d}/K_m$ is generated by $\mathrm{disc}_{L_{m, d}/K_m}(1,  \frac{1+\sqrt{d}}{2})=( \frac{1+\sqrt{d}}{2} - \frac{1-\sqrt{d}}{2})^2=d.$
  	If $d\equiv 3\pmod 4$,  then  $-d\equiv 1\pmod 4$. As we have  $\mathcal{O}_{L_{m, d}}=\mathcal{O}_{L_{m, -d}}$, so by the previous case we easily deduce  the result.
  \end{proof}

  	\begin{proposition}\label{prop deco to 2}
  	Let $m \geq 4$ be an integer and $p$  a prime integer. Then,  $p$ decomposes into the product of  two prime ideals of $K_m$  if and only if   $p\equiv3$ or $5 \pmod 8$.
  \end{proposition}
  \begin{proof}
  	Let $p$ be a rational prime and  $p\mathcal{O}_{K_{4}}=\mathfrak p_1...\mathfrak p_g$ its factorization in $\mathcal{O}_{K_{4}}$. Denote by $f$ the residue degree of $p$ in $K_4$,
  	and by $k$ the positive integer  less  than $16$,  such that  $p\equiv k \pmod{16}$. Then,  by the theorem of  the cyclotomic reciprocity law (see   \cite[Theorem 2.13]{washington1997introduction}),  we have:
   	$${
  		\renewcommand{\arraystretch}{1.8}
  		\setlength{\tabcolsep}{0.5cm}
  		\begin{tabular}{|c|c|c|c|c|c|c|c|c|}
  		\hline
  		$k$	& 1 & 3  & 5  & 7  & 9 & 11 & 13&15 \\
  		\hline
  		$f$& 1&4 & 4 & 2 & 2 &4  & 4 &2  \\
  		\hline
  		$g$& 8&2 & 2 & 4 & 4 &2  & 2 &4  \\
  		\hline
  		\end{tabular} }
  	$$
  	It follows that the  rational primes  that decompose   into the product of two prime ideals of $K_4$   are exactly those which are congruent to $3$ or $5\pmod 8$. So a prime $p$ decomposes into the product of  two prime ideals of $K_m$,  is congruent to $3$ or $5\pmod 8$.
  	For the converse,  assume that  $p\equiv 3$ or $5\pmod 8$ and $p\mathcal{O}_{K_m}=\mathfrak{p}_1 \mathfrak{p}_2$,  for $m\geq4$.  As $K_{m+1}=K_m(\sqrt{\zeta_{2^{m}}})$,  then
  	$\left( \frac{\zeta_{2^{m}}}{\mathfrak p_i}\right)=\left( \frac{\zeta_{2^{m-1}}}{\mathfrak p_i}\right)=-1$. So the result comes by induction.
   \end{proof}


  Let us propose a new simple proof of the following well known result.

  	\begin{theorem}
  	For all $m\geq 2$,  the class number of $K_m=\mathbb Q(\zeta_{ 2^{m}})$ is odd and every unit of $K_m$ is a norm of   an element of $K_{m+1}$.
  \end{theorem}
  \begin{proof}
  	 Note first that $K_{m+1}=K_{m}(\sqrt{\zeta_{2^{m}}})$. Suppose  $h( K_{m})$ is odd for some $m\geq 2$. As $K_{m+1}/K_{m}$ is quadratic extension,  so  the well known ambiguous class number formula (see \text{\cite{Gr}}) implies that
  	$rank_2(K_{m+1})=t_m-1-e_m,$
  	where $e_m$ is defined by $(E_{K_m}:E_{K_m}\cap N_{K_{m+1}/K_{m}}(K_{m+1}^*))=2^{ e_m}$ and $t_m$ is the number of  ramified primes    in $K_{m+1}/K_m$.  Since $2$ is the only rational prime that is ramified in $K_{m+1}$ and it is totally ramified
(in $K_{m+1}$),  hence $t_m=1$. Thus,  $rank_2(K_{m+1})=1-1-e_m=-e_m$. From which we deduce that  $e_m=0$ and $rank_2(K_{m+1})=0.$ So the result comes by induction.
  \end{proof}

  	\section{\textbf{The parity of the class number of the fields $L_{m, d} $}}
  In this section, we investigate  the parity of the  class number of fields $L_{m,  d} $ without  relying on results of Iwasawa theory.
  \begin{theorem}\label{thm the parity}
  	Let $d$  be an odd square-free integer and $m\geq3$  any integer. Then $h(L_{m,  d})$ is odd if and only if $d$ is a prime congruent to $3$ or $5\pmod 8$.
  \end{theorem}
  \begin{proof}
  	Suppose  that $d$ is odd,  and denote by  $ L_{m, d}^*$,   $H_{m, d}$ the genus field and the Hilbert $2$-class field of $L_{m, d}$ respectively. It is  known  that:
  	$$[L_{m, d}^*:\mathbb{Q}]=\prod_{p|\delta_{L_{m, d}}}e(p)\quad \text{ and }\quad  \mathrm{Cl}_2(L_{m, d})=Gal(H_{m, d}/L_{m, d}), $$
  	where $e(p)$ is the ramification index of $p$ in $L_{m, d}$. So  $$[L_{m, d}^*:\mathbb{Q}]=\prod_{p|2d}e(p)=[L_{m, d}^*:L_{m, d}][L_{m, d}:\mathbb{Q}]=2^{m}[L_{m, d}^*:L_{m, d}].$$
  	Since $e(2)=2^{m-1}$ and $e(p)=2$ for any prime divisor $p$ of $d$,  we have
  	$$\prod_{p|d}e(p)=2.[L_{m, d}^*:L_{m, d}].$$
  	Hence,  if $d$ is not a prime,  then $L_{m, d}\subsetneq L_{m, d}^*\subseteq H_{m, d}$
  	and $h_2(L_{m, d})$ is even.
  	
  	Suppose now that $d=p$ is a prime. We distinguish the following four cases: \\
  	$\bullet$ Assume  $d=p\equiv 1\pmod 8$. Set $p=a^2+16b^2=e^2-32f^2$ and $\pi_1=a+4bi$,    $\pi_2= e+4f\sqrt{2}$. As the ramified primes of $K_m$ in $L_{m, d}$ are exactly the prime divisors  of  $p$ in $K_m$,  then the ideals of $L_{m, d}$ generated by $\pi_1$ and $\pi_2$ are squares of ideals of $L_{m, d}$. Note that as $a$ and $e$ are odd, then $a\equiv e \equiv \pm 1\equiv i^2\pmod 4$. It follows that the equation $\pi_j\equiv \xi^2\pmod 4$,  $j=1$ or $2$,  has a solution. So  	$L_1=L_{m, d}(\sqrt{\pi_1})$ and  $L_2=L_{m, d}(\sqrt{\pi_2})$  are two distinct unramified  quadratic extensions of $L_{m, d}$. Thus  $h(L_{m, d})$ is divisible by $4$.  Furthermore,   $\mathrm{Cl}_2(L_{m, d})$ is not trivial and not cyclic.\\
  	\noindent\begin{minipage}{10cm}
  		$\bullet$ Assume now   $d=p\equiv 7\pmod 8$.
  		We prove that $h(L_{m, p})$ is even for all $m\geq 3$ by induction on $m$. If $m=3$,  then $h(L_{3, p})$ is even by \cite[Theorem 4.4]{chemsZkhnin1}. Suppose that $h(L_{m, p})$ is even for some $m\geq 3$. We have
  		$\mathbb{Q}(\sqrt{-p})/\mathbb{Q}$ is unramified at $2$ and $\mathbb{Q}(\zeta_{2^{m}})/\mathbb{Q}$ is totally ramified at $2$,  then $L_{m+1, p}/L_{m, p}$ is a quadratic extension that is ramified at primes over $2$. So $h(L_{m, p})$ divides $h(L_{m+1, p})$,  by Theorem \ref{thm h(k) divise h(L)}. Hence,   $h(L_{m+1, p})$ is even.
  	\end{minipage}
  	\begin{minipage}{4cm}
  		{
  			\tiny
  			\begin{tikzpicture} [scale=1]
  			\node (Q)  at (0, -0.5) {$\mathbb Q$};
  			\node (Q1)  at (-0.3, 0.3) {$\mathbb Q(\sqrt{-p})$};
  			\node (Q2)  at (1.5, 0.5) {$\mathbb Q(\zeta_{2^{m}})$};
  			
  			\node (Q4)  at (0.5, 1.5) {$\mathbb Q(\sqrt{p}, \zeta_{2^{m}})$};
  			\node (Q5)  at (2, 1.5) {$\mathbb Q(\zeta_{2^{m+1}})$};
  			\node (Q7)  at (2, 2.5) {$\mathbb Q(\sqrt{p}, \zeta_{2^{m+1}})$};
  			
  			\draw (Q) --(Q1)  node[scale=0.4, midway, below right]{};
  			\draw (Q) --(Q2)  node[scale=0.4, midway, below right]{};
  			\draw (Q2) --(Q4)  node[scale=0.4, midway, below right]{};
  			\draw (Q4) --(Q7)  node[scale=0.4, midway, below right]{};
  			\draw (Q2) --(Q5)  node[scale=0.4, midway, below right]{};
  			\draw (Q5) --(Q7)  node[scale=0.4, midway, below right]{};
  			\draw (Q1) --(Q4)  node[scale=0.4, midway, below right]{};
  			
  			\end{tikzpicture}}
  	\end{minipage}
  	
  	\noindent$\bullet$ Assume that $d=p\equiv 5\pmod 8$. For $m\geq 3$,   we have $p$ decomposes  into the product of  two prime ideals of $K_m$,  denote by $\mathfrak p_{K_m}$  one of them (such that $\mathfrak p_{K_{m-1}}\subset \mathfrak p_{K_m}$).  Since  $\zeta_{2^{m}}^2=\zeta_{2^{{m-1}}}$,  so the minimal polynomial of $\zeta_{2^{m}}$ over $K_{{m-1}}$ is $X^2-\zeta_{2^{m-1}}$ and $N_{K_{m}/K_{m-1}}(\zeta_{2^{m}})=-\zeta_{2^{m-1}}$.
  	Then
  	$$\left( \frac{\zeta_{2^{m}}, p}{\mathfrak p_{K_m}}\right)=\left( \frac{-\zeta_{ 2^{m-1}}, p}{\mathfrak p_{K_{m-1}}}\right)=\left( \frac{\zeta_{2^{m-1}}, p}{\mathfrak p_{K_{m-1}}}\right)=...=\left( \frac{\zeta_{8}, p}{\mathfrak p_{K_3}}\right)=-1, $$
  	hence  $e_{m, d}\not=0$ and  $rank_2(Cl(L_{m+1, p}))=2-1-e_{m+1, p}=1-e_{m+1, p} =0$. Thus the 2-class group of $L_{m+1, p}$ is trivial and $h(L_{m, d})$ is odd.\\
  	$\bullet$ We treat the case   $d=p\equiv 3\pmod 8$,  similarly to the previous one and we show that $h(L_{m, d})$ is odd. Which achieves the proof.	
  \end{proof}

  \begin{remark}
  	Let  $d$ be a positive square-free integer,
  	$k_\infty$  the cyclotomic $\ZZ_2$-extension of $k=\mathbb{Q}(\sqrt{-1}, \sqrt{d})$,  $k_n$ the nth layer of $k_\infty/k$ and
  	$X_\infty=\varprojlim (\mathrm{Cl}_2(k_n))$,  thus $X_\infty=0$ if and only if $d=p$ is a prime such that $p\equiv 5$ or $3\pmod 8$.	
  \end{remark}

\section{\textbf{The rank of the $2$-class group  of the fields   $L_{m, d}$}}
  	Let $d$ be an odd composite square-free integer of prime divisors congruent to $3$ or $5\pmod 8$ and $m\geq3$   an integer. To state the main theorem of this section,  we need the following result.

  	\begin{lemma}\label{lm computations}
  	Let $m\geq 3$ be an integer and $d$  an odd composite square-free integer. Let $\mathfrak p_{K_{m}}$ denote a prime ideal of $K_m$ dividing $d$.
   	\begin{enumerate}[\rm 1.]
  		\item If  $d=p_1...p_r$,  such that  for all $i$, $p_i\equiv 5  \pmod{8}$ is a prime, then
  		$$\left( \frac{\zeta_{2^m}, d}{\mathfrak p_{K_{m}}}\right)=-1 \text{ and }\left( \frac{\xi_{k, m}, d}{\mathfrak p_{K_{m}}}\right)=1.$$
  		\item If  $d=p_1...p_r$,  such that   for all $i$, $  p_i\equiv 3  \pmod{8}$ is a prime, then
  		$$\left( \frac{\zeta_{2^m}, d}{\mathfrak p_{K_{m}}}\right)=-1 \text{ and }\left( \frac{\xi_{k, m}, d}{\mathfrak p_{K_{m}}}\right)=\left\{  \begin{array}{ccc}
  		-1,  &\text{ if } k\equiv \pm 3 \pmod 8& \\
  		1, & \text{ elsewhere.} &
  		\end{array} \right.$$
  		\item If $d=p_1...p_sp_{s+1}...p_r$,  such that $d$ is   not   prime,   $p_i\equiv 5\pmod{8}$ for  $1\leq i \leq s$ and $p_j\equiv 3\pmod{8}$ for $s+1\leq j \leq r$, then
  		$$\left( \frac{\zeta_{2^m}, d}{\mathfrak p_{K_{m}}}\right)=-1 \text{ and }\left( \frac{\xi_{k, m}, d}{ \mathfrak p_{K_m}}\right)=\left\{\begin{array}{ll}
  		-1,  & \text{ if }   p\equiv 3\pmod{8} \text{ and } k\equiv  \pm 3 \pmod 8\\
  		1,  &\text{ elsewhere,}
  		\end{array}\right.$$
  		where $p$ is the rational prime contained in $\mathfrak p_{K_{m}}$.
  	\end{enumerate}
  \end{lemma}
  \begin{proof}Denote by $\mathfrak p_K$ a prime ideal of a number field $K$ lying over $p$.
  	Each case needs special computations: {
  		\begin{enumerate}[\rm 1.]
  			\item Note that $N_{K_{m}/K_{m-1}}(\zeta_{2^{m}})=-\zeta_{2^{m-1}}$, so
  			$$\left( \frac{\zeta_{2^m}, d}{\mathfrak p_{K_{m}}}\right)=\left( \frac{\zeta_{2^m}, p}{\mathfrak p_{K_{m}}}\right)=\left( \frac{\zeta_{{2^{m-1}}}, p}{\mathfrak p_{{K_{m-1}}}}\right)=...=\left( \frac{\zeta_{8}, p}{\mathfrak p_{{K_{3}}}}\right)=-1,
  			 \text{ and }$$
  $$\left( \frac{1-\zeta_{{2^m}}^k, d}{\mathfrak p_{K_{m}}}\right)=\left( \frac{N_{K_m/K_{m-1}}(1-\zeta_{2^m}^k), d}{\mathfrak p_{{K_{m-1}}}}\right)=...=\left( \frac{1-\zeta_{8}^k, d}{\mathfrak p_{{K_{3}}}}\right)=\left( \frac{1-i^k, d}{\mathfrak p_{K_{2}}}\right).$$
  			
  Thus \\\begin{minipage}{14cm}
  				
  				\begin{eqnarray*}
  					\left( \frac{\xi_{k, m}, d}{\mathfrak p_{K_{m}}}\right)
  					&=&\left( \frac{\zeta_{2^m}^{(1-k)/2}, d}{\mathfrak p_{K_{m}}}\right)\left( \frac{\frac{1-\zeta_{2^m}^k}{1-\zeta_{2^m}}, d}{\mathfrak p_{K_{m}}}\right)\\
  					&=&(-1)^{(1-k)/2}\left( \frac{(1-\zeta_{2^m}^k)(1-\zeta_{2^m}), d}{\mathfrak p_{K_{m}}}\right)\\
  					&=&(-1)^{(1-k)/2}\left( \frac{1-\zeta_{2^m}^k, d}{\mathfrak p_{K_{m}}}\right)\left( \frac{1-\zeta_{2^m}, d}{\mathfrak p_{K_{m}}}\right)\\
  					&=& (-1)^{(1-k)/2}\left( \frac{1-i^k, d}{\mathfrak p_{K_{2}}}\right)\left( \frac{1-i, d}{\mathfrak p_{K_{2}}}\right)\\
  					&=& \left\{\begin{array}{cc}
  						-\left( \frac{1+i, d}{\mathfrak p_{K_{2}}}\right)\left( \frac{1-i, d}{\mathfrak p_{K_{2}}}\right)&\text{ if } k\equiv 3 \pmod4 \\
  						\left( \frac{1-i, d}{\mathfrak p_{K_{2}}}\right)\left( \frac{1-i, d}{\mathfrak p_{K_{2}}}\right)&\text{elsewhere}
  					\end{array} \right.\\
  					&=& \left\{\begin{array}{ll}
  						-\left( \frac{2, d}{\mathfrak p_{K_{2}}}\right)=-\left( \frac{2, p}{\mathfrak p_{K_{2}}}\right)=-\left( \frac{2}{ p}\right)&\text{ if } k\equiv 3 \pmod 4 \\
  						1&\text{elsewhere }
  					\end{array} \right.\\
  					&=&1.	
  			\end{eqnarray*}\end{minipage}
  			\item As in the previous case,  we have 	$\left( \frac{\zeta_{2^m}, d}{\mathfrak p_{K_{m}}}\right)=-1$ and:	\begin{eqnarray}\left( \frac{\xi_{k, m}, d}{\mathfrak p_{K_{m}}}\right)\nonumber
  			&=&\left( \frac{\zeta_{2^m}^{(1-k)/2}, d}{\mathfrak p_{{K_{m}}}}\right)\left( \frac{\frac{1-\zeta_{2^m}^k}{1-\zeta_{2^m}}, d}{\mathfrak p_{{K_{m}}}}\right)\nonumber\\
  			&=&(-1)^{(1-k)/2}\left( \frac{(1-\zeta_{2^m}^k)(1-\zeta_{2^m}), d}{\mathfrak p_{{K_{m}}}}\right)\nonumber\\
  			&=&(-1)^{(1-k)/2}\left( \frac{1-\zeta_{2^m}^k, d}{\mathfrak p_{{K_{m}}}}\right)\left( \frac{1-\zeta_{2^m}, d}{\mathfrak p_{{K_{m}}}}\right)\nonumber\\
  			&=&(-1)^{(1-k)/2} \left( \frac{1-\zeta_{8}^k, d}{\mathfrak p_{{K_{3}}}}\right)\left( \frac{1-\zeta_{8}, d}{\mathfrak p_{{K_{3}}}}\right)\label{eq 5}\\
  			&=&(-1)^{(3-k)/2}\left( \frac{\zeta_{8}^{-1}, d}{\mathfrak p_{{K_{3}}}}\right) \left( \frac{1-\zeta_{8}^k, d}{\mathfrak p_{{K_{3}}}}\right)\left( \frac{1-\zeta_{8}, d}{\mathfrak p_{{K_{3}}}}\right)\nonumber\\
  			&=& \left\{  \begin{array}{cc}
  			\left( \frac{ \varepsilon_2, p}{\mathfrak p_{K_3}}\right), &\hspace{-1.55cm}\text{ if } k\equiv 3 \pmod 8 \\
  			\left( \frac{1+\zeta_{8}, p}{\mathfrak p_{{K_{3}}}}\right)\left( \frac{1-\zeta_{8}, p}{\mathfrak p_{{K_{3}}}}\right), &\text{ if } k\equiv 5 \pmod 8 \;\;\;(\text{see  (\ref{eq 5})})\nonumber \\
  			-\left( \frac{1-\zeta_{8}^{-1}, p}{\mathfrak p_{{K_{3}}}}\right)\left( \frac{1-\zeta_{8}, p}{\mathfrak p_{{K_{3}}}}\right), &\text{ if } k\equiv 7 \pmod 8 \;\;\;(\text{see  (\ref{eq 5})})\\
  			\left( \frac{1-\zeta_{8}, p}{\mathfrak p_{{K_{3}}}}\right)\left( \frac{1-\zeta_{8}, p}{\mathfrak p_{{K_{3}}}}\right), &\text{ if } k\equiv 1 \pmod 8 \;\;\;(\text{see  (\ref{eq 5})})
  			\end{array} \right.
  			\end{eqnarray}
  			\begin{eqnarray}
  			\hspace{2cm}	&=& \left\{  \begin{array}{ccc}
  			-1,  &\hspace{-0.3cm}\text{ if }  k\equiv 3 \pmod 8& \;\;\;\;\;\;(\text{see Lemma \ref{lemma 1 values of symbls p=3 mod 8}}) \\
  			\left( \frac{1-i, p}{\mathfrak p_{{K_{3}}}}\right), &\hspace{-0.3cm}\text{ if }  k\equiv 5 \pmod 8&  \\
  			-\left( \frac{(1-\zeta_{8}^{-1})(1-\zeta_{8}), p}{\mathfrak p_{{K_{3}}}}\right), &\text{ if } k\equiv 7 \pmod 8 \;\;\;\\
  			1, &\text{ if } k\equiv 1 \pmod 8 \;\;\;
  			\end{array} \right.\nonumber
  			\end{eqnarray}
  			\begin{eqnarray*}
  				\hspace{-2cm}	&=& \left\{  \begin{array}{ccc}
  					-1,  &\text{ if } k\equiv 3 \pmod 8& \\
  					\left( \frac{1-i, p}{\mathfrak p_{{K_{2}}}}\right)=\left( \frac{2}{ p}\right), &\text{ if } k\equiv 5 \pmod 8&  \\
  					-\left( \frac{2-\sqrt{2}, p}{\mathfrak p_{{K_{3}}}}\right)=-\left( \frac{2-\sqrt{2}, p}{\mathfrak p_{{\mathbb{Q}(\sqrt2)}}}\right)=-\left( \frac{2}{ p}\right), &\text{ if } k\equiv 7 \pmod 8 &\\
  					1, &\text{ if } k\equiv 1 \pmod 8
  				\end{array} \right.\\
  				&=&\left\{  \begin{array}{clc}
  					-1,  &\text{ if } k\equiv \pm3 \pmod 8, & \\
  					1, &\text{ elsewhere. }   &
  				\end{array} \right.\end{eqnarray*}
  			\item We similarly prove the third  assertion.

  	\end{enumerate} }
  \end{proof}
	\begin{remark}
		Keep the above hypothesis. We have
	\begin{enumerate}[\rm 1.]
		\item $\zeta_{2^m}$ is not a norm in $L_{m, d}/K_m$.
		\item 		$\xi_{k, m}$ is not a norm in $L_{m, d}/K_m$ if and only if  $d$ is divisible by a prime integer congruent to $3\pmod 8$ and $ k\equiv \pm3 \pmod 8$.
	\end{enumerate}
\end{remark}

Now we are able to prove the main result of this section.
\begin{theorem}\label{thm rank}
	Let $d=p_1...p_r$ be an odd composite square-free integer  such that every prime divisor $p_i$ of $d$ is congruent to $3$ or $5$ $\pmod 8$ and $m\geq 3$ is an integer. Then the rank of the $2$-class group of
	$L_{m, d}$ is $2r-2$ or $2r-3$. More precisely,
	$rank_2(Cl(L_{m, d}))=2r-2$ if and only if all   the prime divisors of $d$ are in the same coset $\pmod 8$.
\end{theorem}
\begin{proof}The ring of integers of
	$K_m$ is principal  for  $m\in\{3, 4, 5\}$ (see \cite{masley}). So $h(K_m^+)=1$. By  \cite[Theorem 8.2]{washington1997introduction} and Lemma \ref{lm cyclo units},  the unit group $E_{K_m}$ of $K_m$ is generated by  $\zeta_{ 2^{m}}$ and $\xi_{k, m}=\zeta_{2^m}^{(1-k)/2}\frac{1-\zeta_{2^m}^k}{ 1-\zeta_{2^m}}, $ where $k$ is an odd integer such that  $1< k< 2^{m-1}$. So by the ambiguous class number formula (see \text{\cite{Gr}}) and Proposition \ref{prop decoposition},
	we have  $rank_2(Cl(L_{m, d}))=2r-1-e_{m, d}$.  Let $p$ be a prime divisor of $d$ and $ \mathfrak p_{K_m}$ a prime ideal of $K_m$ lying over $p$.
	If all the prime divisors of $d$ are in the same coset $\pmod 8$,    then by Lemma \ref{lm computations} it is easy to see that
	$E_{K_m}/(E_{K_m}\cap N(L_{m, d}))=\{\overline{1}, \overline{\zeta_{2^m}}\}$. Hence $e_{m, d}=1$.
	
	 Suppose now that   the prime divisors of $d$ are not in the same coset $\pmod 8$. By Lemma \ref{lm computations},  we have $\xi_{k, m}$ is a norm in $L_{m, d}/K_m$,  for all $k\equiv\pm1\pmod8$.

	Let $k\not=k'$ be two odd positive integers such that   $1< k, k'< 2^{m-1}$ and $k, k'\not\equiv\pm1 \pmod 8$. Again by Lemma \ref{lm computations},   we have:
	$$\left( \frac{\xi_{k, m}\xi_{k', m}, d}{\mathfrak p_{K_m}}\right)=1 \text{ for all $\mathfrak p_{K_m}$ of $K_m$,   } $$
	and:
	$$\left( \frac{\zeta_{ 2^{m}}\xi_{k, m}, d}{\mathfrak p_{\mathfrak p_{K_3}}}\right)=-1\; \text{ if } \mathfrak p_{K_m} \text{ is lying over $p\equiv 5\pmod{8}$}.$$
	So $\overline{\xi_{k, m}}=\overline{\xi_{k', m}}$ and $\overline{\xi_{k, m}}\not=\overline{\zeta_{ 2^{m}}}$  in $E_{K_m}/(E_{K_m}\cap N(L_{m, d}))$. Thus $E_{K_m}/(E_{K_m}\cap N(L_{m, d}))= \{\overline{1}, \overline{\zeta_{ 2^{m}}}, \overline{\xi_{k, m}}, \overline{\zeta_{ 2^{m}}\xi_{k, m}}\}$.  Hence   $e_{m, d}=2$. So we have the theorem for $m\in\{3, 4, 5\}$.
	Let $\pi_1= 2$,  $\pi_2=2+\sqrt{2}$, ...,  $\pi_m=2+\sqrt{\pi_{m}}$. Set  $\mathds{k}=\mathbb{Q}(\sqrt{d},  \sqrt{-1})$ and
	$\mathds{k}_1=\mathds{k}(\sqrt{\pi_1})=L_{3, d}$,  $\mathds{k}_2=\mathds{k}(\sqrt{\pi_2})=L_{2, d}$, ...,  $\mathds{k}_m=\mathds{k}(\sqrt{\pi_m})=L_{m, d}$. Thus,  the cyclotomic $\mathbb{Z}_2$-extension   $\mathds{k}_{\infty}$ of $\mathds{k}$ is given by $\cup_{m=0}^{\infty}\mathds{k}_m$. As we have proved Theorem \ref{thm rank} for
	the three layers $\mathds{k}_1$,  $\mathds{k}_2$ and $\mathds{k}_3$,  then Theorem \ref{lm fukuda} achieves the proof.		
\end{proof}

By the previous results, it is easy to get the following interesting theorem.
 \begin{theorem}
 	Let $d$ be  an odd square-free integer and $m\geq 3$ an integer. Suppose that $d$ is not a prime congruent to $7\pmod 8$. Then
 	$\mathrm{Cl}_2(L_{m, d})$ is  cyclic non-trivial if and only if $d=pq$ with $p\equiv 5\pmod 8$ and $q\equiv 3\pmod 8$.
 \end{theorem}
 \begin{proof}
 	In fact,  by \cite[Theorem 5.5]{chemsZkhnin1} we have
 	$\mathrm{Cl}_2(L_{3, d})$  is cyclic non-trivial if and only if $d$  has one of the following forms:
 	\begin{enumerate}[\rm1.]
 		\item $d= q\equiv 7\pmod 8$ is a prime integer.
 		\item $d=qp$,  where   $q\equiv 3\pmod 8$ and $p\equiv 5\pmod 8$ are prime integers.
 	\end{enumerate}
 	Since   $rank_2(\mathrm{Cl}_2(L_{m, d}))\geq  rank_2(\mathrm{Cl}_2(L_{3, d})),  $ then  we get the result by the previous theorem.
 \end{proof}

 \begin{corollary}\label{cor 4.5}
 	Let $d$ be  an odd square-free integer and $m\geq 3$. Suppose that $d$ is not a prime congruent to $7\pmod 8$. Let
 	$k_\infty$ be the cyclotomic $\mathbb{Z}_2$-extension of $k=\mathbb{Q}(\sqrt{-1}, \sqrt{d})$,  $k_n$ the nth layer of $k_\infty/k$ and
 	$X_\infty=\varprojlim (\mathrm{Cl}_2(k_n))$. Thus
 	\begin{enumerate}[\rm 1.]
 		\item $X_\infty$ is cyclic if and only if,  $d=pq$ with $p\equiv 5\pmod 8$ and $q\equiv 3\pmod 8$.
 		\item If $d=pq$ with $p\equiv 5\pmod 8$ and $q\equiv 3\pmod 8$,  then the Iwasawa  $\lambda$-invariant of $k$ equals $0$ or $1$.
 	\end{enumerate}	
 \end{corollary}

 \begin{remark}
 	For any integer $r\geq 0 $ there are infinitely many   imaginary biquadratic number fields $k$ such that $rank(\mathrm{Cl}_2(k_n)))=r$,  $\forall n \geq 1$,  where $k_n$ is the nth layer of $k_\infty /k$.
 \end{remark} 	
  	
  	\section{\textbf{Appendix}}	
  	Let $m\geq 3$ be an integer and $d$  an odd positive square-free integer. Set $\pi_3= 2$,  $\pi_4=2+\sqrt{2}$, ...,  $\pi_m=2+\sqrt{\pi_{m-1}}$ and  $K_m^+=\mathbb{Q}(\sqrt{ \pi_{m}})$.
  	The  maximal real  subfield of $L_{m, d}$ is   $L_{m, d}^+=K_{m}^+(\sqrt{d})$. Note that,  for several cases of positive square-free integers $d$,   the rank of the $2$-class group of $L_{m, d}^+$ is well known in terms of
  	the decomposition of those primes in the cyclotomic tower of $\mathbb{Q}$ that ramify in $\mathbb{Q}(\sqrt{d})/\mathbb{Q}$.
  	In this appendix, we explicitly give   the rank of the $2$-class group of  $L_{m, d}^+$ according to the number of prime divisors of $d$ assuming that   all the prime divisors of $d$ are congruent to $3$ or $5$ $\pmod 8$.
  	
  	\begin{lemma}\label{lm decom in the real field}
  	Let $p$ be a rational prime. Then for all $ m\geq 3$,  $p$ is inert in $K_m^+$ if and only if $p$ is congruent to $3$ or $5\pmod 8$.
  \end{lemma}
  \begin{proof}
  	For $m=3$,     $p$ is inert in $K_3=\mathbb{Q}(\sqrt{2})$ if and only if $p$ is congruent to $3$ or $5\pmod 8$. Thus $p$ is inert in $K_m^+$,  implies that $p$ is congruent to $3$ or $5\pmod 8$.
  	We prove the converse by induction. Suppose that $p$ is inert in $K_m^+$ and show that it is inert in $K_{m+1}^+=\mathbb{Q}(\sqrt{\pi_{m+1}})$. Let   $\mathfrak p$  denote the prime ideal of $K_i^+$ lying over $p$,  for $i\leq m$.  We have $\left( \frac{  {\pi_{m+1}}}{\mathfrak p} \right)=\left( \frac{  {N_{K_m^+/K_{m-1}^+}(\pi_{m+1})}}{\mathfrak p} \right)=\left( \frac{  { 4-{\pi_{m+1}}}}{\mathfrak p} \right)=\left( \frac{  { 2-\sqrt{\pi_{m}}}}{\mathfrak p} \right)=...=\left( \frac{{2}}{ p} \right)=-1.$ It follows  that $p$ is inert in $K_{m+1}^+$.
  \end{proof}

  	\begin{remark}\label{rmk 1}
  	Let  $d=p_1...p_r$  be a square-free integer such that all the prime divisors $p_i$ of $d$ are congruent to $3$ or $5\pmod 8$ and $m\geq 3$.
  	\begin{enumerate}[\rm $\bullet$ ]
  		\item If $d\equiv 1\pmod 4$,  then we have $r$ primes that ramify in  $L_{m, d}^+/K_m^+$,  which are exactly the prime divisors of $d$ in $K_m^+$.
  		\item If $d\not\equiv 1\pmod 4$,  then we have $r+1$ primes that ramify in  $L_{m, d}^+/K_m^+$,  which are exactly the prime of $K_m^+$ lying over $2$ and the prime divisors of $d$ in $K_m^+$.
  	\end{enumerate}	
  \end{remark}

  	\begin{lemma}\label{lm symbols values}
  	Let $m\geq 3$ and $d$  be a positive square-free integer such that all the prime divisors of $d$ are congruent to $3$ or $5\pmod 8$
  	and $\mathfrak p_{K_{m}^+}$ be a prime ideal of $K_m^+$ dividing $d$.
  	Then 	
  	$$\left( \frac{\xi_{k, m}, d}{\mathfrak p_{K_{m}^+}}\right)=\left\{\begin{array}{c c l}-1& \text{ if } & p\equiv 3\pmod 8\text{ and }k\equiv  \pm 3 \pmod 8\\
  	1&     \text{elsewhere},
  	\end{array}\right.$$
  	where $p$ is the rational prime in $ \mathfrak p_{K_{m}^+}$.
  \end{lemma}
  \begin{proof}
  	By lemmas \ref{prop deco to 2} and \ref{lm decom in the real field},       $ \mathfrak p_{K_{m}^+}$ decomposes into the product of two primes of $K_{m}$.   Hence,  the result follows directly from Lemma \ref{lm computations}.
  \end{proof}

  	\begin{theorem}
  	Let $m\geq 3$ and $d=p_1...p_r$   be a positive      square-free integer  such that every prime divisor $p_i$ of $d$ is congruent to $3$ or $5$ $\pmod 8$. Then
  	$$rank_2(L_{m, d}^+)=	\left\{   \begin{tabularx}{\linewidth}{c cX}
  	$r-2$& \text{ if }  &  $d \equiv 1\pmod 4 $ \text{ and } $d$  is divisible by a \\&&prime congruent to $3\pmod 4$,   \\
  	$r-1$& \text{ elsewhere. }  &
  	\end{tabularx}      \right.$$	
  \end{theorem}

  \begin{proof}
  	Similar to the proof of  Theorem \ref{thm rank}.
  \end{proof}

  	\begin{proposition}
  	Let $d=pq$ with    $p\equiv 5\pmod 8$ and $q\equiv 3\pmod 8.$ Then for all $m\geq 3$,  we have:
  	$$h_2(L_{m, d}^+)=2.$$	
  \end{proposition}
  \begin{proof}
  	
  	Let  $\varepsilon_{pq}=a+b\sqrt{pq}$ with $a, b\in \mathbb{Z}$ (resp. $\varepsilon_{2pq}=x+y\sqrt{2pq}$ with $x, y\in \mathbb{Z}$ ) be the fundamental unit of $\mathbb{Q}(\sqrt{pq})$ (resp. $\mathbb{Q}(\sqrt{2pq})$). It is known that $N(\varepsilon_{pq})=N(\varepsilon_{2pq})=1$. We have $a^2-1=b^2pq$ and $x^2-1=y^22pq$.   	So   $a\pm 1$ and $x\pm1$ are not squares in $\mathbb{N}$. In fact,  if $x\pm1$ is a square in $\mathbb{N}$,  then
  	$$	\left\{ \begin{array}{ccc}
  	x\pm1&=&y_1^2\\
  	x\mp1&=&2pqy_2^2,
  	\end{array}\right. $$
  	for some integers $y_1$ and $y_2$ such that $y=y_1y_2$.		
  	So $1=\left(\frac{y_1^2}{p}\right)=\left(\frac{x\pm1}{p}\right)=\left(\frac{x\mp1\pm 2}{p}\right)=\left(\frac{\pm 2}{p}\right)=\left(\frac{ 2}{p}\right)=-1, $ which is absurd. Similarly $a\pm1$ is not a square in $\mathbb{N}$. It follows by \cite[Proposition 3.3]{AZT2016} that $\{\varepsilon_2,  \varepsilon_{pq},  \sqrt{\varepsilon_{pq}\varepsilon_{2pq}}\}$	is a fundamental system of units of 	$L_{3, d}^+=\mathbb{Q}(\sqrt{2}, \sqrt{d})$. Note that by \cite[Corollary 19.7]{connor88},  we have $h_2(pq)=h_2(2pq)=2$.
   	So by Kuruda's class number formula (see \cite{lemmermeyer1994kuroda}),  we obtain
  	$$h_2(L_{3, d}^+)=\frac{1}{4}\cdot 2\cdot h_2(pq)h_2(2pq)h_2(2)=2.$$
  	Thus $h_2(d)=h_2(L_{3, d}^+)=2$. So the result by Theorem \ref{lm fukuda}.
  \end{proof}

  Under the hypothesis of the previous proposition we deduce that the $\mu$-invariant and the $\lambda$-invariant vanishes for such field,  as well we deduce that   the	 $\nu$-invariant equals $1$. For more results on the Iwasawa invariants of real quadratic number fields see \cite{MouhibMouvahhidi}. We close our paper by the following beautiful result:	

  \begin{theorem}
  	Let $n$ be an integer such that every prime $p$ appearing in the decomposition of $n$ with an odd exponent  is congruent to $1\pmod{16} $ or    $7\pmod 8$. Then the equation:
  	$$n=x^2-y^2\zeta_{8}, $$
  	has a solution $(x, y)$ in $\mathbb{Q}(\zeta_{8})\times \mathbb{Q}(\zeta_{8}).$
  \end{theorem}
  \begin{proof}
  	We can suppose that  $n$ is  a positive square-free integer of prime divisors  congruent to $1\pmod{16}$  or    $7\pmod 8$. Let $\mathfrak{p}$ be a prime ideal of $\mathbb{Q}(\zeta_{8})$. If $\mathfrak{p}$ does not divide $n$,  then
  	$\left(\frac{\zeta_8, n}{\mathfrak p}\right)=\left(\frac{n, \zeta_8}{\mathfrak p}\right)=1.$
  	If $\mathfrak{p}$ is lying over a prime divisor $p$ of $n$,  then  we have
  	$\left(\frac{n, \zeta_8}{\mathfrak p}\right)=\left(\frac{\zeta_8, n}{\mathfrak p}\right)=\left(\frac{\zeta_8, p}{\mathfrak p}\right)=1$.
  	So $n$ is a norm in $K'=\mathbb{Q}(\sqrt{\zeta_{8}})=\mathbb{Q}(\zeta_{16}).$ Let $\alpha=x+y\zeta_{16}$ be an  element
  	of $\mathbb{Q}(\zeta_{16})$  such that $n=N_{K'/K}(\alpha)=(x+y\zeta_{16})(x-y\zeta_{16})=x^2-y^2\zeta_{8}$. Which gives the result.
  \end{proof}

  \section*{\textbf{Acknowledgment}}
  We  would like  to take this opportunity to sincerely thank professor Radan Kučera for his remarks that helped to complete the  proof of   Theorem \ref{thm the parity}.

  We also thank the referee for his/her suggestions that held us to improve our paper.

\end{document}